\documentclass[12pt]{amsart}
\usepackage[cp1251]{inputenc}
\usepackage[T2A]{fontenc}
\usepackage[english]{babel}
\usepackage{amsmath,amsfonts,amssymb}
\usepackage{geometry}
\usepackage{amsmath, amsthm, amscd, amsfonts, amssymb, graphicx, color}
\usepackage[bookmarksnumbered, colorlinks, plainpages]{hyperref}

\numberwithin{equation}{section}
\newtheorem{theorem}{Theorem}[section]

\theoremstyle{remark}

\theoremstyle{definition}

\begin{document}

\title[Parameter-dependent boundary problems in H\"older spaces]{One-dimensional parameter-dependent boundary-value problems in H\"older spaces}


\author[H. Masliuk]{Hanna Masliuk}
\address{National Technical University of Ukraine “Igor Sikorsky Kyiv Polytechnic Institute”, Peremohy Avenue 37, 03056, Kyiv-56, Ukraine}
\email{masliukgo@ukr.net}


\author[V. Soldatov]{Vitalii Soldatov}
\address{Institute of Mathematics, National Academy of Sciences of Ukraine, Tere\-shchen\-kivska Str. 3, 01004 Kyiv-4, Ukraine}
\email{soldatovvo@ukr.net, soldatov@imath.kiev.ua}




\begin{abstract}
We study the most general class of linear boundary-value problems for systems of $r$-th order ordinary differential equations whose solutions range over the complex H\"older space $C^{n+r,\alpha}$, with $0\leq n\in\mathbb{Z}$ and $0<\alpha\leq1$. We prove a constructive criterion under which the solution to an arbitrary parameter-dependent problem from this class is continuous in $C^{n+r,\alpha}$ with respect to the parameter. We also prove a two-sided estimate for the degree of convergence of this solution to the solution of the corresponding nonperturbed problem.
\end{abstract}

\maketitle

\section{Introduction}\label{1}

Parameter-dependent systems of ordinary differential equations often appear in mathematics and its applications. The main question concerning these systems is under which conditions we may pass to the limit in the solution to the corresponding Cauchy problem or a relevant boundary-value problem. As to the Cauchy problem, such conditions were found by Gikhman \cite{Gikhman1952}, Krasnosel'skii and S.~Krein \cite{KrasnoselskiiKrein1955}, Kurzweil and Vorel \cite{KurzweilVorel1957CMJ}. In the case of linear differential systems, more fine conditions were obtained by Levin \cite{Levin1967dan}, Opial \cite{Opial1967}, Reid \cite{Reid1967}, and Nguyen The Hoan \cite{NguenTheHoan1993}.

As compared with the Cauchy problem, parameter-dependent boundary-value problems are less investigated, which is connected with a great diversity of boundary conditions. Nevertheless, there are important results concerning some broad classes of linear boundary-value problems.
The most known among the letter is the class of so-called general linear boundary-value problems for systems of first-order differential equations. Their solutions are supposed to be absolutely continuous on a compact interval $[a,b]$, whereas the boundary condition is given in the form $By=q$ where $B:C([a,b],\mathbb{R}^{m})\to\mathbb{R}^{m}$ is an arbitrary continuous linear operator (here, $y$ is a solution, and $m$ is the number of differential equations in the system). Kiguradze  \cite{Kiguradze1975, Kiguradze1987, Kiguradze2003} and Ashordia \cite{Ashordia1996} found conditions under which the solutions to these problems are continuous in the normed space $C([a,b],\mathbb{R}^{m})$ with respect to the parameter.
Recently \cite{KodliukMikhailetsReva2013, MikhailetsChekhanova2015} these results were refined and extended to complex-valued functions and systems of higher-order differential equations.

Of late years, Mikhailets and his disciples introduced and investigated the broadest classes of linear boundary-value problems for linear differential systems whose solutions range over a chosen complex Sobolev space or space $C^{l}$ of $l$ times continuously differential functions (see   \cite{MikhailetsReva2008DAN9, GnypKodlyukMikhailets2015UMJ, KodliukMikhailets2013JMS} and  \cite{MikhailetsChekhanova2014DAN7, Soldatov2015UMJ} respectively). The boundary conditions for these problems are posed in the form $By=q$ where $B$ is an arbitrary continuous linear operator acting from the chosen function space to the finite-dimensional complex space. These problems are called generic with respect to this space. Mikhailets and his disciples found constructive conditions under which the solutions to the parameter-dependent generic problem are continuous in the chosen space with respect to the parameter. It is turned out that these conditions are not only sufficient but also necessary \cite{GnypMikhailetsMurach2017EJDE, MikhailetsMurachSoldatov2016, MikhailetsMurachSoldatov2016MFAT}. These results were applied to the investigation of multipoint boundary-value problems \cite{Kodliuk2012Dop11}, Green's matrices of boundary-value problems \cite{KodliukMikhailetsReva2013, MikhailetsChekhanova2015}, to the spectral theory of differential operators with distributional coefficients \cite{GoriunovMikhailets2010MN87.2, GoriunovMikhailetsPankrashkin2013EJDE, GoriunovMikhailets2012UMJ63.9}.

In this connection, it is interesting to study generic boundary-value problems with respect to fractional analogs of the Sobolev spaces and the spaces $C^{l}$. The H\"older spaces are such analogs for $C^{l}$. As to systems of first order differential equations, Mikhailets, Murach and Soldatov \cite{MikhailetsMurachSoldatov2016, MurachSoldatov2016Dop10} introduced and investigated generic boundary-value problems with respect to the H\"older space $C^{n+1,\alpha}$, with $0\leq n\in\mathbb{Z}$ and $\alpha\in(0,1]$. These authors found a constructive criterion for the solution of a such parameter-dependent problem to be continuous in the H\"older space in the parameter. They also proved a two-sided estimate for the degree of convergence of this solution to the solution of the corresponding nonperturbed problem. Quite later on, Masliuk \cite{Masliuk2017UMJ} investigated generic boundary-value problems with respect to the H\"older space $C^{n+r,\alpha}$ for systems of $r$-th order differential equations. She found sufficient conditions under which the solutions to these problems are continuous in $C^{n+r,\alpha}$ with respect to the parameter.

The purpose of the present paper is to prove that these conditions are also necessary. The proof is complicated by the absence of an explicit description of the duals of the H\"older spaces (c.f. \cite{GnypMikhailetsMurach2017EJDE, MikhailetsMurachSoldatov2016MFAT}).
Besides, we will obtain a two-sided estimate for the degree of convergence of these solutions.

Note that for the fractional analogs of Sobolev spaces---Slobodetsky spaces---generic boundary-value problems are investigated in \cite{Gnyp2016UMJ, MasliukMikhailets2018UMJ}.

\section{Main results}\label{6sec2}

We arbitrarily choose a compact interval $[a,b]\subset\mathbb{R}$, integers $m\geq1$, $r\geq2$, and $n\geq0$ and a real number $\alpha$ subject to the condition $0<\alpha\leq1$. We use the complex H\"older spaces
$(C^{n,\alpha})^{m}:=C^{n,\alpha}([a,b],\mathbb{C}^{m})$ and $(C^{n,\alpha})^{m\times
m}:=C^{n,\alpha}([a,b],\mathbb{C}^{m\times m})$ of indexes $n$ and $\alpha$. They consist respectively of all vector-valued or matrix-valued functions whose entries belong to the H\"older space $C^{n,\alpha}:=C^{n,\alpha}([a,b],\mathbb{C})$ of scalar functions on $[a,b]$, with the vectors being of $m$ entries and with the matrixes being of $m\times m$ type. Recall that the space $C^{n,\alpha}$ consists, by definition, of all $n$ times continuously differentiable functions $x:[a,b]\to\mathbb{C}$ such that
$$
\|x\|_{n,\alpha}':=\sup_{a\leq t_1<t_2\leq b}
\frac{|x^{(n)}(t_2)-x^{(n)}(t_1)|}{|t_2-t_1|^{\alpha}}<\infty.
$$
This space is endowed with the norm
$$
\|x\|_{n,\alpha}:=\sum_{j=0}^{n}\max_{a\leq t\leq b}|x^{(j)}(t)|+
\|x\|_{n,\alpha}'
$$
and is a Banach algebra with respect to a certain norm which is equivalent to $\|\cdot\|_{n,\alpha}$. The norms in the Banach spaces $(C^{n,\alpha})^{m}$ and $(C^{n,\alpha})^{m\times
m}$ are defined to equal the sums of the norms of all the entries in $C^{n,\alpha}$ and are denoted by $\|\cdot\|_{n,\alpha}$ as well. It will be clear from context to which H\"older space of indexes $n$ and $\alpha$ (scalar, vector or matrix-valued functions) the designation $\|\cdot\|_{n,\alpha}$ relates.

Let a real number $\varepsilon_{0}>0$ be fixed, and let a real parameter $\varepsilon$ range over the interval $[0,\varepsilon_{0})$.
We investigate a parameter-dependent liner boundary-value problem of the form
\begin{gather}\label{1syste}
L(\varepsilon)y(t,\varepsilon)\equiv
y^{(r)}(t,\varepsilon)+\sum_{j=1}^r
A_{r-j}(t,\varepsilon)y^{(r-j)}(t,\varepsilon)=f(t,\varepsilon),
\quad a\leq t\leq b,\\
B(\varepsilon)y(\cdot,\varepsilon)=c(\varepsilon). \label{1kue}
\end{gather}
For every fixed $\varepsilon\in[0,\varepsilon_{0})$, the solution $y(\cdot,\varepsilon)$ to the problem is considered in the class $(C^{n+r,\alpha})^{m}$. We suppose that $A_{r-j}(\cdot, \varepsilon)\in (C^{n,\alpha})^{m\times m}$ for each
$j\in\{1,...,r\}$ and that $f(\cdot,\varepsilon)\in(C^{n,\alpha})^{m}$. Thus, \eqref{1syste} is a system of $m$ scalar linear $r$-th order differential equations given on $[a,b]$. Note we do not assume $A_{r-j}(\cdot,\varepsilon)$ to have any regularity in $\varepsilon$.
As to the boundary condition \eqref{1kue}, we suppose that $B(\varepsilon)$ is an arbitrary continuous linear operator
\begin{equation*}
B(\varepsilon):(C^{n+r,\alpha})^{m}\to \mathbb C^{rm}
\end{equation*}
and that $c(\varepsilon)\in\mathbb C^{rm}$. Naturally, we interpret vectors and vector-valued functions as columns.

The boundary condition \eqref{1kue} is the most general for the system
\eqref{1syste} because the right-hand side $f(\cdot,\varepsilon)$ of the system runs through the whole space $(C^{n,\alpha})^{m}$ if and only if the solution $y(\cdot,\varepsilon)$ to the system runs through the whole space $(C^{n+r,\alpha})^{m}$. We therefore call the boundary-value problem \eqref{1syste}, \eqref{1kue} generic with respect to the space $C^{n+r,\alpha}$.

With this problem, we associate the continuous linear operator
\begin{equation}\label{6.LBe}
(L(\varepsilon),B(\varepsilon)):(C^{n+r,\alpha})^{m}\to
(C^{n,\alpha})^{m}\times\mathbb{C}^{rm}.
\end{equation}
According to \cite[Theorem~1]{Masliuk2017UMJ}, this operator is Fredholm of zero index for every $\varepsilon\in[0,\varepsilon_0)$.

Let us consider the following four

\medskip

\noindent{\bf Limit Conditions} as $\varepsilon\to 0+$:
\begin{itemize}
\item [(I)]$A_{r-j}(\cdot,\varepsilon)\to
A_{r-j}(\cdot,0)$ in $(C^{n,\alpha})^{m\times m}$ for each $j\in\{1,...\,,r\}$;
\item [(II)] $B(\varepsilon)y\to  B(0)y$ in $\mathbb{C}^{rm}$ for every $y\in (C^{n+r,\alpha})^{m}$;
\item [(III)] $f(\cdot,\varepsilon)\to f(\cdot,0)$ in $(C^{n,\alpha})^{m}$;
\item [(IV)] $c(\varepsilon)\to c(0)$ in $\mathbb C^{rm}$.
\end{itemize}

We also consider the so-called

\medskip

\noindent{\bf Condition (0).} The homogeneous boundary-value problem
\begin{equation}\label{6.LB0}
L(0)y(t,0)=0,\quad a\leq t\leq b,\qquad B(0)y(\cdot,0)=0
\end{equation}
has only the trivial solution.

\medskip

Let us give our

\medskip

\noindent{\bf Basic Definition.} We say that the solution to the boundary-value problem \eqref{1syste}, \eqref{1kue} depends continuously on the parameter $\varepsilon$ at $\varepsilon=0$ if the following two conditions are satisfied:
\begin{itemize}
\item [$(\ast)$] There exists a positive number $\varepsilon_{1}<\varepsilon_{0}$ that this problem has a unique solution $y(\cdot,\varepsilon)\in (C^{n+r,\alpha})^{m}$ for arbitrarily chosen $\varepsilon\in[0,\varepsilon_{1})$, $f(\cdot,\varepsilon)\in (C^{n,\alpha})^{m}$, and $c(\varepsilon)\in \mathbb C^{rm}$.
\item [$(\ast\ast)$] It follows from Limit Conditions (III) and (IV) that
\begin{equation} \label{6.gu}
y(\cdot,\varepsilon)\to  y(\cdot,0)\;\;\mbox{in}\;\;
(C^{n+r,\alpha})^{m}\;\;\mbox{as}\;\;\varepsilon\to 0+.
\end{equation}
\end{itemize}

Let us formulate the main result of the paper.

\medskip

\noindent{{\bf Main Theorem.} \it The solution to the boundary-value problem \eqref{1syste}, \eqref{1kue} depends continuously on the parameter $\varepsilon$ at $\varepsilon=0$ if and only if this problem satisfies Condition~\textup{(0)} and Limit Conditions \textup{(I)} and
\textup{(II)}.\rm

\medskip

We supplement this result with a two-sided estimate of the error $\|y(\cdot,0)-y(\cdot,\varepsilon)\|_{n+r,\alpha}$ of the solution $y(\cdot,\varepsilon)$ via
its discrepancy
\begin{equation}
d_{n,\alpha}(\varepsilon):=
\|L(\varepsilon)y(\cdot,0)-f(\cdot,\varepsilon)\|_{n,\alpha}+
|B(\varepsilon)y(\cdot,0)-c(\varepsilon)|.
\end{equation}
Here, recall, $\|\cdot\|_{n,\alpha}$ stands for the norm in the space
$(C^{n,\alpha})^{m}$, whereas $|\cdot|$ denotes the norm in  $\mathbb{C}^{rm}$. Besides, we interpret $y(\cdot,0)$ as an approximate solution to the problem \eqref{1syste}, \eqref{1kue}.

\begin{theorem}\label{th-1}
Let the boundary-value problem \eqref{1syste}, \eqref{1kue} satisfy Condition~\textup{(0)} and Limit Conditions \textup{(I)} and~\textup{(II)}. Then there exist positive numbers $\varepsilon_{2}<\varepsilon_{1}$, $\varkappa_{1}$, and
$\varkappa_{2}$ such that
\begin{equation}\label{6.bound}
\varkappa_{1}\,d_{n,\alpha}(\varepsilon)
\leq\|y(\cdot,0)-y(\cdot,\varepsilon)\|_{n+r,\alpha}\leq
\varkappa_{2}\,d_{n,\alpha}(\varepsilon)
\end{equation}
for every $\varepsilon\in(0,\varepsilon_{2})$. Here, the numbers
$\varepsilon_{2}$, $\varkappa_{1}$, and $\varkappa_{2}$ are independent of the functions $y(\cdot,0)$ and $y(\cdot,\varepsilon)$.
\end{theorem}

Thus, the error and discrepancy of the solution to the problem \eqref{1syste}, \eqref{1kue} are of the same degree as $\varepsilon\to0+$.

In the case of $r=1$, these theorems were proved in  \cite{MikhailetsMurachSoldatov2016} (see also \cite{MurachSoldatov2016Dop10}). Of course, the condition $\varepsilon\to0+$ is not essential in these theorems. We may replace it with the condition $\varepsilon\to\varepsilon_{0}$ provided that $\varepsilon_{0}$ is a limit point of the range of values of the parameter~$\varepsilon$.

\section{Proofs of the main results}\label{6sec3}

\begin{proof}[Proof of Main Theorem]
Masliuk \cite[Theorem~3]{Masliuk2017UMJ} proved the sufficiency of Condition~(0) and Limit Conditions (I) and (II) for the problem \eqref{1syste}, \eqref{1kue} to satisfy Basic Definition. Let us prove their necessity. Suppose that this problem satisfies Basic Definition. Then the problem obviously meets Condition~(0). It remains to prove that the problem satisfies Limit Conditions (I) and (II). We divide our reasoning into three steps.

\emph{Step~$1$.} Let us prove that the problem \eqref{1syste},
\eqref{1kue} satisfies Limit Condition~(I). To this end, we canonically reduce the system \eqref{1syste} to a certain system of first order differential equations.  We put
\begin{gather}\label{def-x}
x(\cdot,\varepsilon):=\mathrm{col}\bigl(y(\cdot,\varepsilon),y'(\cdot, \varepsilon),\ldots,y^{(r-1)}(\cdot, \varepsilon)\bigr)\in(C^{n+1,\alpha})^{rm},\\
g(\cdot,\varepsilon):=\mathrm{col}\bigl(0,f(\cdot,
\varepsilon)\bigr) \in(C^{n,\alpha})^{rm},\notag
\end{gather}
and
$$
A(\cdot,\varepsilon):=\begin{pmatrix}
O_{m} & -I_{m} & O_{m} & \ldots & O_{m} \\
O_{m} & O_{m} & -I_{m} & \ldots & O_{m} \\
\vdots & \vdots & \vdots & \ddots & \vdots \\
O_{m} & O_{m} & O_{m} & \ldots & -I_{m} \\
A_{0}(\cdot,\varepsilon) & A_{1}(\cdot,\varepsilon) & A_{2}(\cdot,\varepsilon) & \ldots & A_{r-1}(\cdot,\varepsilon)\\
\end{pmatrix}\in(C^{n,\alpha})^{rm \times rm}.
$$
Here, $I_m$ (respectively, $O_m$) denotes the identity (resp., zero) $m\times m$ matrix. If $y(\cdot,\varepsilon)\in(C^{n+r,\alpha})^{m}$ is a solution to the system \eqref{1syste}, then $x(\cdot,\varepsilon)$ is a solution to the system
\begin{equation*}
x'(t,\varepsilon)+A(t,\varepsilon)x(t,\varepsilon)=g(t,\varepsilon),
\quad a\leq t\leq b.
\end{equation*}

Consider the following matrix boundary value problem:
\begin{gather}\notag
Y^{(r)}(t,\varepsilon)+\sum_{j=1}^r
A_{r-j}(t,\varepsilon)Y^{(r-j)}(t,\varepsilon)=O_{m\times rm},
\quad a\leq t\leq b,\\
[B(\varepsilon)Y(\cdot,\varepsilon)]=I_{rm}. \label{6.matrix-bound-cond}
\end{gather}
Here,
$$
Y(\cdot,\varepsilon):=
\bigl(y_{j,k}(\cdot,\varepsilon)\bigr)_
{\substack{j=1,\ldots,m\\k=1,\ldots,rm}}
$$
is an unknown $m\times rm$ matrix-valued function with entries from $C^{n+r,\alpha}$. Besides, $O_{m\times rm}$, of course, denotes the zero $m\times rm$ matrix, and
$$
[B(\varepsilon)Y(\cdot,\varepsilon)]:=
\left(B(\varepsilon)
\begin{pmatrix}
y_{1,1}(\cdot,\varepsilon)\\
\vdots \\
y_{m,1}(\cdot,\varepsilon)\\
\end{pmatrix}
\;\ldots\; B(\varepsilon)\begin{pmatrix}
y_{1,rm}(\cdot,\varepsilon)\\
\vdots \\
y_{m,rm}(\cdot,\varepsilon)\\
\end{pmatrix}\right).
$$
This problem is a collection of $rm$ boundary-value problems \eqref{1syste}, \eqref{1kue} whose right-hand sides do not depend of $\varepsilon$. Therefore, this problem has a unique solution $Y(\cdot,\varepsilon)$ for every $\varepsilon\in[0,\varepsilon_{1})$
due to condition ($\ast$) of Basic Definition. Moreover, owing to
condition ($\ast\ast$) of this definition, each
\begin{equation}\label{Y-convergence}
y_{j,k}(\cdot,\varepsilon)\to y_{j,k}(\cdot,0)
\;\;\mbox{in}\;\;C^{n+r,\alpha}\;\;\mbox{as}\;\;\varepsilon\to0+.
\end{equation}

Given $k\in\{1,\ldots,rm\}$ and $\varepsilon\in[0,\varepsilon_{1})$, we define a vector-valued function $x_{k}(\cdot,\varepsilon)\in(C^{n+1,\alpha})^{rm}$ by formula
\eqref{def-x} in which we replace $x(\cdot,\varepsilon)$ with $x_{k}(\cdot,\varepsilon)$ and take
$$
y(\cdot,\varepsilon):=
\mathrm{col}(y_{1,k}(\cdot,\varepsilon),\ldots,
y_{m,k}(\cdot,\varepsilon)).
$$
Let $X(\cdot,\varepsilon)$ denote the matrix-valued function from $(C^{n+1,\alpha})^{rm\times rm}$ such that its $k$-th column is $x_{k}(\cdot,\varepsilon)$ for each $k\in\{1,\ldots,rm\}$. This function satisfies the matrix differential equation
\begin{equation}\label{6.matrix-eq}
X'(t,\varepsilon)+A(t,\varepsilon)X(t,\varepsilon)=O_{rm},
\quad a\leq t\leq b.
\end{equation}
Therefore, $\det X(t,\varepsilon)\neq0$ whenever $t\in[a,b]$, for otherwise the columns of the matrix-valued function $X(\cdot,\varepsilon)$ and, hence, of $Y(\cdot,\varepsilon)$ would be linearly dependent on $[a,b]$, contrary to~\eqref{6.matrix-bound-cond}. Owing to \eqref{Y-convergence}, we have the convergence $X(\cdot,\varepsilon)\to X(\cdot,0)$ in the Banach algebra $(C^{n+1,\alpha})^{rm\times rm}$ as $\varepsilon\to0+$. Hence, $(X(\cdot,\varepsilon))^{-1}\to(X(\cdot,0))^{-1}$ in this algebra. Therefore, in view of \eqref{6.matrix-eq}, we conclude that
\begin{equation*}
A(\cdot,\varepsilon)=-X'(\cdot,\varepsilon)(X(\cdot,\varepsilon))^{-1}\to
-X'(\cdot,0)(X(\cdot,0))^{-1}=A(\cdot,0)
\end{equation*}
in $(C^{n,\alpha})^{rm\times rm}$ as $\varepsilon\to0+$. Thus, the problem \eqref{1syste}, \eqref{1kue} satisfies Limit Condition~(I). Specifically,
\begin{equation}\label{6.bound-norm-A}
\|A_{r-j}(\cdot,\varepsilon)\|_{n,\alpha}=O(1)
\;\;\mbox{as}\;\;\varepsilon\to0+
\;\;\mbox{for each}\;\;j\in\{1,\ldots,r\}.
\end{equation}

\emph{Step~$2$.} Let us prove that
\begin{equation}\label{6.10}
\|B(\varepsilon)\|=O(1)\quad\mbox{as}\quad\varepsilon\to0+.
\end{equation}
Here, $\|\cdot\|$ stands for the norm of a bounded operator from $(C^{n+r,\alpha})^m$ to $\mathbb{C}^{rm}$.

Suppose the contrary; then there exists a sequence
$(\varepsilon^{(k)})_{k=1}^{\infty}\subset(0,\varepsilon_{1})$ such that
\begin{equation}\label{6.11}
\varepsilon^{(k)}\to0\quad\mbox{and}\quad
\|B(\varepsilon^{(k)})\|\to\infty\quad\mbox{as}\quad
k\to\infty,
\end{equation}
with $\|B(\varepsilon^{(k)})\|\neq0$ whenever $k\geq1$. For every integer $k\geq1$, we choose a function
$w_{k}\in(C^{n+r,\alpha})^{m}$ that satisfies the conditions
\begin{equation}\label{6.12}
\|w_{k}\|_{n+r,\alpha}=1\quad\mbox{and}\quad
|B(\varepsilon^{(k)})w_{k}|\geq\frac{1}{2}\,\|B(\varepsilon^{(k)})\|.
\end{equation}
Besides, we put
\begin{gather*}
y(\cdot,\varepsilon^{(k)}):=
\|B(\varepsilon^{(k)})\|^{-1}\,w_{k}\in(C^{n+r,\alpha})^{m},\\
f(\cdot,\varepsilon^{(k)}):=
L(\varepsilon^{(k)})\,y(\cdot,\varepsilon^{(k)})\in(C^{n,\alpha})^{m},\\
c(\varepsilon^{(k)}):=
B(\varepsilon^{(k)})\,y(\cdot,\varepsilon^{(k)})\in\mathbb{C}^{rm}.
\end{gather*}
Owing to \eqref{6.11} and \eqref{6.12}, we have the convergence
\begin{equation}\label{6.z-convergence}
y(\cdot,\varepsilon^{(k)})\to0\;\;\mbox{in}\;\;
(C^{n+r,\alpha})^{m}\;\;\mbox{as}\;\;k\to\infty.
\end{equation}
Hence,
\begin{equation}\label{6.f-convergence}
f(\cdot,\varepsilon^{(k)})\to0\;\;\mbox{in}\;\;
(C^{n,\alpha})^{m}\;\;\mbox{as}\;\;k\to\infty
\end{equation}
because the problem \eqref{1syste}, \eqref{1kue} satisfies Limit Condition~(I) (this was proved on step~1). Besides, according to
\eqref{6.12}, we conclude that
$$
1/2\leq|c(\varepsilon^{(k)})|\leq1\quad
\mbox{whenever}\quad 1\leq k\in\mathbb{Z}.
$$
Hence, there exists a subsequence $(c(\varepsilon^{(k_p)}))_{p=1}^{\infty}$ of $(c(\varepsilon^{(k)}))_{k=1}^{\infty}$ and a nonzero vector
$c(0)\in\nobreak\mathbb{C}^{rm}$ such that
\begin{equation}\label{6.q-convergence}
c(\varepsilon^{(k_p)})\to c(0) \;\;\mbox{in}\;\;\mathbb{C}^{rm}
\;\;\mbox{as}\;\; p\to\infty.
\end{equation}

For every integer $p\geq1$, the function
$y(\cdot,\varepsilon^{(k_p)})\in(C^{n+r,\alpha})^{m}$ is a unique solution to the boundary-value problem
\begin{gather*}
L(\varepsilon^{(k_p)})y(t,\varepsilon^{(k_p)})=f(t,\varepsilon^{(k_p)}), \quad a\leq t\leq b,\\
B(\varepsilon^{(k_p)})y(\cdot,\varepsilon^{(k_p)})=c(\varepsilon^{(k_p)}).
\end{gather*}
Owing to \eqref{6.f-convergence} and \eqref{6.q-convergence} and  condition $(\ast\ast)$ of Basic Definition, we conclude that the function $y(\cdot,\varepsilon^{(k_p)})$
converges to the unique solution $y(\cdot,0)$ of the boundary-value problem
\begin{equation*}
L(0)y(t,0)=0,\quad a\leq t\leq b, \qquad B(0)y(\cdot,0)=c(0),
\end{equation*}
with convergence being in $(C^{n+r,\alpha})^{m}$ as
$k\to\infty$. But $y(\cdot,0)\equiv0$ due to
\eqref{6.z-convergence}. This contradicts the boundary condition
$B(0)y(\cdot,0)=c(0)$, in which $c(0)\neq0$. Thus, our assumption is false, which proves the required property \eqref{6.10}.

\emph{Step~$3$.} Using the results of the previous steps, we will prove here that the problem \eqref{1syste}, \eqref{1kue} satisfies Limit Condition~(II). According to \eqref{6.bound-norm-A} and \eqref{6.10}, there exist numbers $\varkappa'>0$ and $\varepsilon'\in(0,\varepsilon_{1})$ such that
\begin{equation}\label{6.L(e)B(e)-bound}
\|(L(\varepsilon),B(\varepsilon))\|\leq\varkappa'\quad \mbox{for every}\quad\varepsilon\in[0,\varepsilon').
\end{equation}
Here, $\|(L(\varepsilon),B(\varepsilon))\|$ denotes the norm of the bounded operator~\eqref{6.LBe}. We arbitrarily choose a vector-valued function $y\in(C^{n+r,\alpha})^{m}$ and set
$f(\cdot,\varepsilon):=L(\varepsilon)y$ and
$c(\varepsilon):=B(\varepsilon)y$ for every
$\varepsilon\in[0,\varepsilon')$. Hence,
\begin{equation}\label{6.z-equation-L(e)B(e)}
y=(L(\varepsilon),B(\varepsilon))^{-1}
(f(\cdot,\varepsilon),c(\varepsilon))\quad\mbox{for every}\quad
\varepsilon\in[0,\varepsilon').
\end{equation}
Here, $(L(\varepsilon),B(\varepsilon))^{-1}$ denotes the inverse of the operator \eqref{6.LBe}; the latter is invertible due to condition ($\ast$) of Basic Definition.

Using \eqref{6.L(e)B(e)-bound} and \eqref{6.z-equation-L(e)B(e)}, we obtain the following inequalities for every $\varepsilon\in(0,\varepsilon')$:
\begin{align*}
&|B(\varepsilon)y-B(0)y|\leq
\bigl\|(f(\cdot,\varepsilon),c(\varepsilon))-
(f(\cdot,0),c(0))\bigr\|_{(C^{n,\alpha})^{m}\times\mathbb{C}^{rm}}\\
&=\bigl\|(L(\varepsilon),B(\varepsilon))
(L(\varepsilon),B(\varepsilon))^{-1}
\bigl((f(\cdot,\varepsilon),c(\varepsilon))-
(f(\cdot,0),c(0))\bigr)\bigr\|_{(C^{n,\alpha})^{m}\times\mathbb{C}^{rm}}\\
&\leq\varkappa'\,\bigl\|(L(\varepsilon),B(\varepsilon))^{-1}
\bigl((f(\cdot,\varepsilon),c(\varepsilon))-
(f(\cdot,0),c(0))\bigr)\bigr\|_{n+r,\alpha}\\
&=\varkappa'\,\bigl\|(L(0),B(0))^{-1}(f(\cdot,0),c(0))-
(L(\varepsilon),B(\varepsilon))^{-1}(f(\cdot,0),c(0))\bigr\|_{n+r,\alpha}.
\end{align*}
The latter norm vanishes as $\varepsilon\to0+$ according to condition ($\ast\ast$) of Basic Definition. Therefore, $B(\varepsilon)y\to B(0)y$ in $\mathbb{C}^{rm}$ as $\varepsilon\to0+$. Since $y\in(C^{n+r,\alpha})^{m}$ is arbitrary chosen, we conclude that the boundary-value problem \eqref{1syste}, \eqref{1kue} satisfies Limit Condition~(II).
\end{proof}

\begin{proof}[Proof of Theorem $\ref{th-1}$]
Let us prove the left-hand side of the required estimate~\eqref{6.bound}. It follows from Limit Conditions (I) and (II) that the bounded operator \eqref{6.LBe} converges strongly to $(L(0),B(0))$
as $\varepsilon\to0+$. Hence, according to the Banach--Steinhaus theorem,        there exist numbers $\varkappa'>0$ and $\varepsilon'\in(0,\varepsilon_{1})$ that the norm of this operator satisfies \eqref{6.L(e)B(e)-bound}. Therefore,
\begin{align*}
d_{n,\alpha}(\varepsilon)&=
\|L(\varepsilon)(y(\cdot,0)-y(\cdot,\varepsilon))\|_{n,\alpha}+
|B(\varepsilon)(y(\cdot,0)-y(\cdot,\varepsilon))|\\
&\leq\varkappa'\,\|y(\cdot,0)-y(\cdot,\varepsilon)\|_{n+r,\alpha}
\end{align*}
for every $\varepsilon\in(0,\varepsilon')$, which gives the left-hand side of the estimate \eqref{6.bound}.

Let us now prove the right-hand side of this estimate. According to Main Theorem, the boundary-value problem \eqref{1syste}, \eqref{1kue} satisfies Basic Definition. Hence, the operator \eqref{6.LBe} is invertible for every $\varepsilon\in[0,\varepsilon_1)$. Moreover, its inverse  $(L(\varepsilon),B(\varepsilon))^{-1}$ converges strongly to $(L(0),B(0))^{-1}$ as $\varepsilon\to0+$. Indeed, it follows from condition $(\ast\ast)$ of Basic Definition that
\begin{equation*}
(L(\varepsilon),B(\varepsilon))^{-1}(f,c)=:y(\cdot,\varepsilon)\to
y(\cdot,0):=(L(0),B(0))^{-1}(f,c)
\end{equation*}
in $(C^{n+r,\alpha})^{m}$ as $\varepsilon\to0+$ for all $f\in(C^{n,\alpha})^{m}$ and $c\in\mathbb{C}^{rm}$. By Banach--Steinhaus  theorem, there exist numbers $\varkappa_{2}>0$ and $\varepsilon_{2}\in(0,\varepsilon')$ that the norm of the inverse of \eqref{6.LBe} satisfies the condition $\|(L(\varepsilon),B(\varepsilon))^{-1}\|\leq\varkappa_{2}$ whenever $\varepsilon\in[0,\varepsilon_{2})$. Hence,
\begin{align*}
\|y(\cdot,0)-y(\cdot,\varepsilon)\|_{n+r,\alpha}
&=\|(L(\varepsilon),B(\varepsilon))^{-1}(L(\varepsilon),B(\varepsilon))
(y(\cdot,0)-y(\cdot,\varepsilon))\|_{n+r,\alpha}\leq\\
&\leq\varkappa_{2}\,\|(L(\varepsilon),B(\varepsilon))
(y(\cdot,0)-y(\cdot,\varepsilon))\|_
{(C^{n,\alpha})^{m}\times\mathbb{C}^{rm}}=
\varkappa_{2}\,d_{n,\alpha}(\varepsilon)
\end{align*}
for every $\varepsilon\in[0,\varepsilon_{2})$. We have obtained the right-hand side of the required estimate~\eqref{6.bound}.
\end{proof}

\section{Concluding remarks}

The pair of Limit Conditions (I) and (II) is equivalent to the fact that operator \eqref{6.LBe} converges strongly to $(L(0),B(0))$ as $\varepsilon\to0+$. This equivalence follows from Theorem~\ref{th-2} proved below. Therefore, Main Theorem asserts specifically that this strong convergence and the invertibility of the limit operator $(L(0),B(0))$ implies the invertibility of \eqref{6.LBe} whenever $0<\varepsilon\ll1$  and the strong convergence of the inverse of \eqref{6.LBe} to $(L(0),B(0))^{-1}$ as $\varepsilon\to0+$. Note that, for arbitrary continuous operators acting between infinite-dimensional Banach spaces and depending on $\varepsilon$, the analog of this implication is not true.

\begin{theorem}\label{th-2}
Limit Condition $\mathrm{(I)}$ is equivalent to each of the following two conditions:
\begin{itemize}
\item[(c1)] $\|L(\varepsilon)-L(0)\|\to0$ as $\varepsilon\to0+$;
\item[(c2)] $L(\varepsilon)y\to L(0)y$ in $(C^{n,\alpha})^{m}$ as $\varepsilon\to0+$ for every $y\in(C^{n+r,\alpha})^{m}$.
\end{itemize}
Here, $\|\cdot\|$ denotes the norm of the continuous linear operator
\begin{equation*}
L(\varepsilon):(C^{n+r,\alpha})^{m}\to (C^{n,\alpha})^{m}.
\end{equation*}
\end{theorem}

\begin{proof}
Since $\mathrm{(c1)}\Rightarrow\mathrm{(c2)}$, it remains to prove that  Limit Condition (I) implies condition (c1) and that condition (c2) implies  Limit Condition (I). Let us prove the first of these implications. Choosing $y\in(C^{n+r,\alpha})^{m}$ arbitrarily, we write
\begin{align*}
\|L(\varepsilon)y-L(0)y\|_{n,\alpha}&\leq
\sum_{j=1}^{r}\bigl\|(A_{r-j}(\cdot,\varepsilon)
-A_{r-j}(\cdot,0))\,y^{(r-j)}\bigr\|_{n,\alpha}\\
&\leq c_{1}\sum_{j=1}^{r}
\|A_{r-j}(\cdot,\varepsilon)-A_{r-j}(\cdot,0)\|_{n,\alpha}\,
\|y^{(r-j)}\|_{n,\alpha}\\
&\leq c_{2}\|y\|_{n+r,\alpha}\sum_{j=1}^{r}
\|A_{r-j}(\cdot,\varepsilon)-A_{r-j}(\cdot,0)\|_{n,\alpha}.
\end{align*}
Here, $c_{1}$ and $c_{2}$ are certain positive numbers that do not depend on~$y$. Therefore, if Limit Condition (I) is satisfied, then
\begin{equation*}
\|L(\varepsilon)-L(0)\|\leq c_{2}\sum_{j=1}^{r}
\|A_{r-j}(\cdot,\varepsilon)-A_{r-j}(\cdot,0)\|_{n,\alpha}\to0.
\end{equation*}
In the proof all limits are considered provided that $\varepsilon\to0+$. Thus, Limit Condition (I) implies condition (c1).

Let us now prove that condition (c2) implies  Limit Condition (I). Suppose that condition (c2) is satisfied. Then
\begin{equation}\label{k_m-c}
Z^{(r)}+\sum_{j=1}^{r}A_{r-j}(\cdot,\varepsilon)Z^{(r-j)}\to
Z^{(r)}+\sum_{j=1}^{r}A_{r-j}(\cdot,0)Z^{(r-j)}
\end{equation}
for every matrix-valued function $Z\in(C^{n+r,\alpha})^{m\times m}$. This and the next limits hold true in the space $(C^{n,\alpha})^{m\times m}$. Putting $Z(\cdot)\equiv I_{m}$ in \eqref{k_m-c}, we obtain the convergence $A_{0}(\cdot,\varepsilon)\to A_{0}(\cdot,0)$. Choose an integer $k\in\{0,\ldots,r-2\}$ arbitrarily and assume that $A_l(\cdot,\varepsilon)\to A_l(\cdot,0)$ for each $l\in\{0,\ldots,k\}$. Let us prove that $A_{k+1}(\cdot,\varepsilon)\to A_{k+1}(\cdot,0)$. Putting $Z(t)\equiv t^{k+1}I_m$ in \eqref{k_m-c}, we obtain the convergence
\begin{equation*}
(k+1)!\,A_{k+1}(\cdot,\varepsilon)+
\sum_{l=0}^{k}A_{l}(\cdot,\varepsilon)Z^{(l)}\to
(k+1)!\,A_{k+1}(\cdot,0)+\sum_{l=0}^{k}A_{l}(\cdot,0)Z^{(l)}.
\end{equation*}
Here, by the last assumption,
\begin{equation*}
\sum_{l=0}^{k}A_{l}(\cdot,\varepsilon)Z^{(l)}\to
\sum_{l=0}^{k}A_{l}(\cdot,0)Z^{(l)}.
\end{equation*}
Hence, $A_{k+1}(\cdot,\varepsilon)\to A_{k+1}(\cdot,0)$. We have proved by the induction that Limit Condition~(I) is satisfied.
\end{proof}

\end{document}